\documentclass[12pt]{amsart}
\usepackage{amsmath}
\usepackage{amsthm}
\usepackage{amscd}

\numberwithin{equation}{section}

\newtheorem{theo}{Theorem}[section]
\newtheorem{prop}[theo]{Proposition}
\newtheorem{lemm}[theo]{Lemma}
\newtheorem{cor}[theo]{Corollary}
\newtheorem{claim}[theo]{Claim}

\theoremstyle{definition}
\newtheorem{defi}[theo]{Definition}
\theoremstyle{remark}
\newtheorem{rem}[theo]{Remark}
\theoremstyle{definition}

\newcommand{\vol} {\mathrm{vol_{\it{X} | \it{V} }}}
\newcommand{\vo} {\mathrm{vol_{\it{X} | \it{V^{'}} }}}
\newcommand{\voll} {\mathrm{vol_{\widetilde{ \it{X}} |\it{\widetilde V} }}}
\newcommand{\T} {({T\vert _{V_{\reg}}})_{\ac}^d} 
\newcommand{\Tm}{ ((\mu ^{*}T)\vert _{\widetilde V})_{\ac}^{d}} 
\newcommand{\e}{\varepsilon }
\newcommand{\del}{\delta }
\newcommand{\Vr}{V_{\mathrm{reg}}}
\newcommand{\OX} {\mathcal{O}_{X}}
\newcommand{\OV} {\mathcal{O}_{V}}
\newcommand{\OW} {\mathcal{O}_{W}}

\newcommand{\ac}{\mathrm{ac}}
\newcommand{\sing}{\mathrm{sing}}

\newcommand{\A}{\alpha}
\newcommand{\B}{\beta }
\newcommand{\ddbar}{dd^c}
\newcommand{\reg}{ \mathrm{reg} }

\newcommand{\V}[1] {\mathrm{vol_{ \it{X} | \it{#1} }}}
\newcommand{\Vm}[1] {\mathrm{vol_{ \it{M} | \it{#1} }}}

\newcommand{\I}{\mathcal{I}}
\newcommand{\J}{\mathcal{J}}

\newcommand{\CC}{\mathbb{C}}
\newcommand{\RR}{\mathbb{R}}
\newcommand{\QQ}{\mathbb{Q}}
\newcommand{\ZZ}{\mathbb{Z}}

\begin{document}

\title[RESTRICTED VOLUMES AND DIVISORIAL ZARISKI DECOMPOSITIONS]
{RESTRICTED VOLUMES AND\\ DIVISORIAL ZARISKI DECOMPOSITIONS}
\author{Shin-ichi MATSUMURA}

\address{Graduate School of Mathematical Sciences, University of Tokyo, 3-8-1 Komaba,
Tokyo, 153-8914, Japan.}

 \email{{\tt
shinichi@ms.u-tokyo.ac.jp, 
shinichi@sci.kagoshima-u.ac.jp, mshinichi0@gmail.com
}}

\thanks{Research of the author supported in part 
by JSPS 
the Grant-in-Aid for Scientific Research (KAKENHI No. 23-7228). }


\maketitle

\begin{abstract}
We give a relation between the existence of a Zariski decomposition and 
the behavior of the restricted volume of a big divisor on a smooth (complex) projective variety. 
Moreover, we give an analytic description of the restricted volume 
in the line of Boucksom's work.
It enables us to define the restricted volume of a transcendental class 
on a compact K\"ahler manifold in natural way.
The relation can be extended to a transcendental class.
\end{abstract}

\section{Introduction}
Throughout this paper, $X$ denotes a smooth projective variety of dimension $n$, 
$D$ a (big) divisor on $X$ and 
$V$ an irreducible subvariety of dimension $d$ on $X$, 
unless otherwise mentioned.
Then the restricted volume of $D$ along $V$ is defined to be 
\begin{equation*}
\V{V}(D) := \limsup_{k \to \infty} \dfrac
{\dim H^{0}(X|V, \mathcal{O}_{X}(kD)) }{k^d/d!}. 
\end{equation*}
Here we denote by 
$H^{0}(X|V, \mathcal{O}_{X}(kD)) \subseteq H^{0}(V, \mathcal{O}_{V}(kD))$ 
the space of global sections of $\OV(kD)$ on $V$ that can be extended to $X$. 
Roughly speaking, the restricted volume measures the number of sections of $\OV(kD)$ 
which can be extended to $X$.
The notion of the restricted volume first appeared in \cite{Tsu06}. 
The restricted volume has many applications in various situations
(see \cite{HM06}, \cite{Ta06}).
The properties of the restricted volume 
are studied in \cite{ELMNP09}, \cite{BFJ09} and so on.

On the other hand, it is an important problem to determine  
when $D$ admits a Zariski decomposition. 
Here a decomposition $D=P+N$ is said to be a Zariski decomposition, if 
$P$ is a nef $\RR$-divisor and $N$ is an effective $\RR$-divisor 
such that the following map is an isomorphism for any positive integer $k>0$:
\begin{equation*}
H^{0}(X, \OX(\lfloor kP \rfloor)) \longrightarrow  H^{0}(X, \OX(kD)).
\end{equation*}
This map is the natural map induced by the section $e_{k}$, 
where $e_{k}$ is the standard section of the effective divisor $\lceil kN \rceil$.
Here $\lfloor G \rfloor$ (resp. $\lceil G \rceil$) denotes 
the divisor defined by the round-downs 
(resp. the round-ups) of the coefficients of an $\RR$-divisor $G$.

When $D$ is an ample divisor, the restricted volume $\V{V}(D)$ of $D$ along $V$ is equal to 
the self-intersection number $(D^{d}\cdot V)$ of $D$ on $V$. 
Therefore the restricted volume $\V{V}(D)$ along $V$ 
depends only on the first Chern class (the numerical class) of $D$ when $D$ is ample.  
In general, 
the restricted volume has the same property if $V$ is not contained in 
the augmented base locus $\mathbb{B}_{+}(D)$ of $D$ 
(see \cite[Theorem A]{ELMNP09}). 
The augmented base locus of $D$ is a subvariety on $X$ 
which measures how far $D$ is from ample divisors 
(see \cite[Section 1]{ELMNP06} for the precise definition and the properties).

It is natural to 
ask whether the restricted volume $\V{V}{(D)}$ of $D$ 
depends only on the numerical class of $V$.
In \cite{BFJ09}, the question is affirmatively answered 
when the codimension  of $V$ is one.
In this paper, we give a necessary and sufficient condition for $D$, that
the restricted volume $\V{V}(D)$ of $D$ depends only on the numerical class of $V$.
The condition is related to the existence of a Zariski decomposition of $D$ as follows:

\begin{theo}\label{ main}
Let $D$ be a big divisor on a smooth projective variety $X$. 
Then the following conditions are equivalent.

$(1)$ $D$ admits a Zariski decomposition.

$(2)$ $\V{V}(D) = \V{V^{'}}(D)$ holds for any pair of subvarieties $V$ and $V^{'}$ on $X$
such that $V\equiv V^{'}$ and $V, V^{'} \not \subseteq \mathbb{B}_{+}(D)$. 

$(3)$ $\V{C}(D) = \V{C^{'}}(D)$ holds for any pair of curves $C$ and $C^{'}$ on $X$
such that $C\equiv C^{'}$ and $C, C^{'} \not \subseteq \mathbb{B}_{+}(D)$. 
\end{theo}
\begin{rem}
It is sufficient for the proof of Theorem \ref{ main} 
to show that condition (1) (resp. (3)) implies condition (2) (resp. (1)) 
since condition (2) clearly leads to condition (3). 
\end{rem}\vspace{-1mm}
When subvarieties $V$ and $V^{'}$ are numerically equivalent, 
we write $V \equiv V^{'}$. 
Condition (2) means that the restricted volume $\V{V}(D)$
of $D$ depends only on the numerical class of $V$. 
Theorem \ref{ main} implies  
that the restricted volumes along some numerically equivalent subvarieties  
are different when $D$ does not admit a Zariski decomposition.

When $V$ is the ambient space $X$, the restricted volume of $D$ is equal to 
the usual volume $\mathrm{vol}_{X}(D)$ of $D$.
The usual volume has been studied by several authors.
The general theory is presented in details in \cite{La}.
In his paper \cite{Bou02}, 
Boucksom gave an analytic description of  the usual volume 
with positive curvature currents 
which represent the first Chern class of $D$,
by using a result of Fujita on the approximation of Zariski decompositions and the singular holomorphic Morse inequalities. 
In other words, Boucksom expressed the usual volume of $D$ 
in terms of the first Chern class of $D$.

The restricted volume $\V{V}(D)$ 
along $V$ depends only on the first Chern class $c_{1}(D)$ of $D$ 
if $V$ is not contained in the augmented base locus $ {\mathbb{B}}_{+}(D)$ of $D$. 
Then Boucksom's description for the usual volume 
can be generalized to the restricted volume as follows:

\begin{theo}\label{ main2}
Let $D$ be a big divisor on a smooth projective variety $X$. 
Assume that $V$ is not contained in 
the augmented base locus $ {\mathbb{B}}_{+}(D)$ of $D$.
Then the restricted volume of $D$ along $V$ satisfies the following equality:
\begin{equation*}
\vol (D) = \sup_{T\in c_1(D)} \int_{V_{\reg}} { (T\vert_{V_{\reg}})_{\ac}^d} 
\end{equation*}
where $T$ ranges among positive $(1,1)$-currents 
with analytic singularities in $c_{1}(D)$ 
whose singular loci do not contain $V$.
\end{theo}
Here we denote by $T\vert_{V_{\reg}}$ the restriction of $T$ to the regular locus $V_{\reg}$ of $V$ and 
by $(T\vert_{V_{\reg}})_{\ac}$ the absolutely continuous part of $T\vert_{V_{\reg}}$ 
(see subsection 2.4 for the precise definition).
Theorem \ref{ main2} enables us to define the restricted volume of 
a transcendental class  on a compact K\"ahler manifold in natural way.

\begin{defi} \label{ def}
Let $W$ be an irreducible analytic subset of dimension $d$ on a compact K\"ahler manifold $M$ 
and let $\A$ a class in $H^{1,1}(M , \RR)$.
Assume that $W$ is not contained in the non-K\"ahler locus $E_{nK}(\A)$ of $\A$. 
Then \textit{the restricted volume} of $\A$ along $W$ is defined to be
\begin{equation*}
\mathrm{vol}_{M|W}(\A) := \sup_{T\in \A} \int_{W_{\reg}}{({T|_{W_{\reg}}})_{\ac}^d} 
\end{equation*}
where $T$ ranges among positive $(1,1)$-currents with analytic singularities in $\A$ 
whose singular loci do not contain $W$.
\end{defi}

Here the non-K\"ahler locus is an analytic counterpart of the augmented base locus 
(see \cite[Definition 3.14]{Bou04} for the precise definition of the non-K\"ahler locus). 
When $\A$ is the first Chern class 
of some divisor $D$, the non-K\"ahler locus $E_{nK}(\A)$ coincides 
with the augmented base locus $\mathbb{B}_{+}(D)$.
For this extended definition, the properties of the usual restricted volume hold. 
For example, the continuity, log concavity, Fujita's approximations and so on 
(see subsection 4.2).
Moreover, an analogue of Theorem \ref{ main} holds for the extended definition as follows.
The proof gives another proof of Theorem \ref{ main} by using analytic methods 
(see subsection 4.3).
\begin{theo}\label{ main3}
Let $\A$ be a big class in $H^{1,1}(X, \RR)$ on a smooth projective variety $X$.
Then the following conditions are equivalent.

$(1)$ $\A$ admits a Zariski decomposition.

$(2)$ $\V{V}(\A) = \V{V^{'}}(\A)$ holds for any pair of subvarieties $V$ and $V^{'}$ on $X$
such that $V\equiv V^{'}$ and $V, V^{'} \not \subseteq E_{nK}(\A)$.

$(3)$ $\V{C}(\A) = \V{C^{'}}(\A)$ holds for any pair of curves $C$ and $C^{'}$ on $X$
such that $C\equiv C^{'}$ and $C, C^{'} \not \subseteq E_{nK}(\A)$.

\end{theo}
Here we say that a big class admits a Zariski decomposition if the positive 
part of its divisorial Zariski decomposition is nef (see subsection 2.7).
When $\A$ is the first Chern class of some divisor $D$ 
(that is, $\A$ is contained in the N\'eron-Severi space), 
a Zariski decomposition of $\A$ 
coincides with that of  $D$.
However, 
a class $\A$ is not necessarily contained in the N\'eron-Severi space of $X$ 
even if $X$ is projective. 
Therefore Theorem \ref{ main3} is essentially stronger statement than Theorem \ref{ main}.

\section{Preliminaries} 
In this section, we prepare for the proofs. 
The propositions in this section may be known facts. 
However we give comments or references for the readers' 
convenience. 
Throughout this section, $M$ denotes a compact K\"ahler manifold of dimension $n$.

\subsection{Currents}
Since $M$ is a K\"ahler manifold, 
$H^{p,p}(M,\CC)$ is identified with the quotient of the space of $d$-closed $(p,p)$-currents modulo the $\ddbar$-exact currents.
For our purpose, the case of $p=1$ is important. 
We say that a function $\varphi$ is a  potential function of a $(1,1)$-current $T$ 
if $T = \ddbar \varphi$. 
Notice that a (local) potential function is uniquely determined 
modulo the pluriharmonic functions.
If $T$ is $d$-closed, we can locally take a potential function of $T$.   
 A $d$-closed $(1,1)$-current is said to have analytic (resp. algebraic) singularities 
(along the subscheme $V(\I)$ defined by an ideal sheaf $\I$), 
if its potential function $\varphi$ can be locally written as 
\begin{equation*}
\varphi =\frac{c}{2}\log(|f_1|^2+...+|f_k|^2) + v
\end{equation*}
for some $c \in \RR_{>0}$ (resp. $c \in \QQ_{>0}$), 
where $f_1,\dots,f_k$ are local generators of $\I$ and $v$ is a smooth function. 
Then $V(\I)$ is called the singular locus of the current. 

\subsection{The pull-backs of $(1,1)$-currents} 
Let us confirm the definition of 
the pull-back of a $d$-closed $(1,1)$-current by a holomorphic map. 
Let $T$ be a $d$-closed $(1,1)$-current on $M$ and let 
$f : Z\to M$ a holomorphic map from a complex manifold $Z$ to $M$. 
Assume that the image of $Z$ by $f$ is not contained 
in the polar set of a potential function of $T$.
Then we can define the pull-back of $T$ by $f$ as follows:  
Since $T$ is $d$-closed, 
we can locally take a potential function $\varphi $ of $T$. 
Then the pullback of $T$ is (locally) defined to be $f^{*}T:=\ddbar f^{*} \varphi $. 
It determines a global $d$-closed $(1,1)$-current on $Z$
since $\ddbar f^{*} \varphi$ does not depend 
on the choice of a local potential function $\varphi$. 
In particular, we can restrict a $d$-closed $(1,1)$-current to a submanifold  
if the submanifold is not contained in the polar set of its potential function.
Notice that the pull-back $f^{*}T$ is also positive if $T$ is positive.

\subsection{Multiplier ideal sheaves and Skoda's lemma.}
In this paper, we often use the description 
of the restricted volume with the multiplier ideal sheaf 
which was proved in \cite{ELMNP09}. 
We denote by $\I(T)$ the multiplier ideal sheaf 
associated to a $d$-closed $(1,1)$-current $T$. 
That is, $\I(T)$ is the sheaf of germs of holomorphic functions $f$ 
such that $|f|^{2} e^{-2\varphi}$ is locally integrable, 
where $\varphi$ is a local potential function of $T$. 
(Note that this definition does not depend on the choice of a local potential function.) 
See \cite{DEL00}, \cite{Dem} for more details. 
Skoda's Lemma gives a relation between the Lelong number of $T$ and 
the multiplier ideal sheaf $\I(T)$. 
Here the Lelong number $\nu(T,x)$ of 
an almost positive $(1,1)$-current $T = \ddbar \varphi$ at $x$ 
is defined by $\displaystyle \nu(T,x) : = \liminf_{z \to x} \dfrac{\varphi (z) }{\log |z-x|}$ 
where $z$ is a local coordinate centered at $x$.

\begin{lemm}$($\cite{Sko72}$).$ \label{ Skoda} Let $\varphi$ be a potential function of 
an almost positive current $T$.

$\mathrm{(a)}$ If $\nu(T,x) <1$, 
then $e^{-2 \varphi}$ is integrable in a neighborhood of $x$. 
In other words, 
the stalk $\I (T)_{x}$ of $\I (T)$ at $x$ is equal to the stalk $\mathcal{O}_{M,x}$ of 
the structure sheaf of $M$ at $x$.

$\mathrm{(b)}$ If $\nu(T,x) \geq n+s$ for some positive integer $s$, then $e^{-2 \varphi}\geq C |z-x|^{-2n-2s}$ in 
a  neighborhood of $x$. In particular, we have 
$\I (T)_{x} \subseteq  \frak{m}_{M,x}^{s+1}$, 
where $\frak{m}_{M, x}$ is the maximal ideal of $\mathcal{O}_{M,x}$.
\end{lemm}

\subsection{Lebesgue decompositions}
A positive current $T$ can be  locally regarded as 
a $(1,1)$-form with measure coefficients.
Thus it admits the Lebesgue decomposition into the absolutely continuous part and 
the singular part with respect to the Lebesgue measure. 
Therefore we obtain the decomposition $T = T_{\ac} + T_{\sing}$, 
where $T_{\ac}$ (resp.  $T_{\sing}$) is 
the absolutely continuous part (resp. the singular part) of $T$. 
This decomposition is globally determined 
thanks to the uniqueness of the Lebesgue decomposition. 
Now $T_{\ac}$ is considered as a $(1, 1)$-form with ${L^{1}_{loc}}$-function coefficients. 
Thus we can define the product $T_{\ac}^{k}$ of $T_{\ac}$ almost everywhere. 
We have $T_{\ac} \geq \gamma $ if 
$T \geq \gamma $ for some smooth $(1,1)$-form $\gamma$. 
In particular, the absolutely continuous part $T_{\ac}$ is positive 
if $T$ is a positive $(1,1)$-current
(see \cite[Section 2.3]{Bou02} for more details). 

\subsection{Approximations of currents}
Let $T= \theta + \ddbar \varphi$ be a $(1,1)$-current in a class $\A \in  H^{1,1}(M, \RR)$, 
where $\theta$ is a smooth $(1,1)$-form in $\A$ and 
$ \varphi$ is an $L^{1}$-function on $M$.
We assume that $T\geq \gamma $ holds for a smooth form $\gamma $. 
Fix a K\"ahler form $\omega$ on $M$.  
Then we can approximate $T$ by smooth forms in the following sense:

\begin{theo}\label{ app-dem82}$($\cite[TH\'EOR\`EME 9.1]{Dem82}$)$.
There exists a decreasing sequence of smooth functions $\varphi _{k}$ converging to $\varphi $ 
such that if we set $T_{k} = \theta + \ddbar \varphi _{k} \in \A$, we have

\ \ $\mathrm{(a)}$\ \ $T_{k} \longrightarrow  T$ weakly and $T_{k} \longrightarrow  T_{\ac} $ almost everywhere on $M$.

\ \ $\mathrm{(b)}$\ \ $T_{k} \geq \gamma - C\lambda _{k} \omega $, 
where $C$ is a positive constant depending only on $(M,\omega )$, 
and $\{ \lambda _{k} \}_{k=1}^{\infty}$ is a decreasing sequence of continuous functions  such that 
$\lambda _{k}(x) \searrow \nu(T,x) $ for all $x \in M$.

\end{theo}
Roughly speaking, Theorem \ref{ app-dem82} says that it is possible to smooth a given current $T$ insides 
the class $\A$, but only with the loss of positivity controlled by the Lelong numbers of $T$. 
By the proof of Theorem \ref{ app-dem82} in \cite{Dem82}, we may add the following property to Theorem \ref{ app-dem82}. 
(Recall that $T_{k}$ is obtained from $T$ by convolution with a regularized kernel.)

\textit{
$\mathrm{(c)}$\ \ If $T$ is smooth on a given open set $U$ of $M$, 
then $T_{k}$ converges to $T$ in $C^{\infty}(U)$. 
}\\
\ \ The following theorem asserts 
that it is possible to approximate a given current with currents 
with analytic singularities. 
There is a loss of positivity but it is arbitrary small. 

\begin{theo}\label{ app-bou02} $($\cite{Dem92}, \cite[Theorem 2.4]{Bou02}$).$
There exists a sequence of functions $\varphi _{k}$ with analytic singularities 
converging to $\varphi $ such that if we set $T_{k} = \theta + \ddbar \varphi _{k} \in \A$, we have

\ \ $\mathrm{(a')}$\ \ $T_{k} \longrightarrow  T$ weakly and $T_{k,\ac} \longrightarrow  T_{\ac} $ almost everywhere.

\ \ $\mathrm{(b')}$\ \ $T_{k} \geq \gamma -\e_{k}\omega $, where $\e_{k}$ is a positive number converging to zero.

\ \ $\mathrm{(c')}$\ \ The Lelong number $\nu(T_{k},x)$ increases to $\nu(T,x)$ uniformly with respect to $x \in M$.
\end{theo}

In the proof of \cite[Theorem 2.4]{Bou02}, 
the convergence $T_{k,\ac} \longrightarrow  T_{\ac} $ in $\mathrm{(a')}$ 
was obtained from only property 
$\mathrm{(a)}$ in Theorem \ref{ app-dem82}.
Therefore we may add the following property $\mathrm{(d')}$ thanks to 
property $\mathrm{(c)}$.

\textit{
\ $\mathrm{(d')}$\ \  If $T$ is smooth on a given open set $U$ of $M$, 
then $T_{k, \ac}$ converges to $T_{\ac}$ in $C^{\infty}(U)$.}
\\
\ \ It yields the following corollary.

\begin{cor}\label{ app-bou02co}
Let $W$ be an irreducible analytic subset on $M$. 
Assume that $T\vert_{W_{\reg}}$ is smooth except some analytic set on $W_{\reg}$. 
Then $T_{k}$ in Theorem \ref{ app-bou02} satisfies the following property:
\begin{equation*}
\big( T_{k} \vert_{W_{\reg}} \big)_{\ac} 
\longrightarrow {\big( T \vert_{W_{\reg}} \big) }_{\ac}\ \ \ \ 
\text{ almost everywhere on }W_{\reg}.
\end{equation*}
\end{cor}

\subsection{Asymptotic invariants of base loci.}
In this subsection, we collect the definitions and properties of 
the augmented base locus and the restricted base locus of a divisor. 
See Definition 1.2, 1.12 in \cite{ELMNP09} for more details. 

\begin{defi}
Let $L$ be an $\RR$-divisor on $X$. \\
(1) When $L$ is a $\QQ$-divisor, \textit{the stable base locus} 
$\mathbb{B}(L)$ of $L$ 
is defined by  
\begin{equation*}
\mathbb{B}(L):= \bigcap_{k} {\rm{Bs}}(|kL|) 
\end{equation*}
where $k$ runs through all positive integers such that $kL$ is 
a $\ZZ$-divisor. \\
(2) \textit{The augmented base locus} $\mathbb{B}_{+}(L)$ 
of $L$ is defined by  
\begin{equation*}
\mathbb{B}_{+}(L):= \bigcap_{L \equiv A + E} {\rm{Supp}}(E) 
\end{equation*}
where the intersection is taken over all decomposition $L \equiv A+E$, 
$A$ and $E$ are $\RR$-divisors such that $A$ is ample and $E$ is effective.\\
(3) \textit{The restricted base locus} $\mathbb{B}_{-}(L)$ 
of $L$ is defined by 
\begin{equation*}
\mathbb{B}_{-}(L):= \bigcup_{A } \mathbb{B}(L+A) 
\end{equation*}
where the union is taken over all ample $\RR$-divisors $A$ such that 
$L+A$ is a $\QQ$-divisor. 
\end{defi}

Let us recall the definitions of the non-K\"ahler locus and the non-nef locus 
of a class $\A \in H^{1,1}(M, \RR)$ 
(see definition 3.3, 3.17 in \cite{Bou04}). 

\begin{defi}
(1) Assume $\A$ is a big class (that is, it possesses a K\"ahler current). 
Then \textit{the non-K\"ahler locus} $E_{nK}(\A)$ of $\A$ is defined to be 
\begin{equation*}
E_{nK}(\A) := \bigcap_{T \in \A} E_{+}(T) 
\end{equation*}
where $T$ ranges among the K\"ahler currents in $\A$. 
Here $E_{+}(T)$ denotes $\{x \in M\ |\ \nu(T,x)>0  \}$. \\
(2) Assume $\A$ is a pseudo-effective class (that is, it possesses a positive current). 
Then \textit{the non-nef locus} $E_{nn}(\A)$ of $\A$ is defined to be 
\begin{equation*}
E_{nn}(\A) :=  \{x \in M\ |\ \nu(\alpha ,x)>0  \}. 
\end{equation*}
Here $\nu(\alpha ,x)$ is $\sup_{\e >0} \nu(T_{\min, \e}, x)$,  
where $T_{\min, \e}$ is a current with minimal singularities in $\A + \e \{\omega \}$ 
and $\{ \omega \}$ is the class of a K\"ahler form $\omega$ on $M$. 
\end{defi}

In this paper, we need the following properties of these base loci. 
We state the properties of the non-K\"ahler (non-nef) locus 
without the proofs (but we give the references). 
Note that the non-K\"ahler (resp. non-nef) locus of $\A$ coincides with  
the augmented  (resp. restricted) base locus of $L$ 
when $\A$ is the first Chern class of some divisor $L$. 
Thus, the augmented (restricted) base locus also satisfies the following properties.  
\begin{prop} \label{ pro-bs}
$(1)$ $($\cite[Section 5]{ELMNP09}$).$ 
Given a class $\A \in H^{1,1}(M, \RR)$, 
we have $E_{nK}(\beta) \subset E_{nK}(\A)$ 
for every class $\beta$ in a sufficiently small open neighborhood of 
$\A \in H^{1,1}(M, \RR)$. \\
$(2)$ $($\cite[Theorem 3.17]{Bou04}$).$
If $\A$ is big, 
there is a K\"ahler current $S$ in $\A$ with analytic singularities 
such that $E_{+}(S) = E_{nK}(\A)$. \\
$(3)$ $($\cite[Proposition 3.6]{Bou04}$).$
If $\A$ is big, we have 
\begin{equation*}
E_{nn}(\A) = \{ x \in M\ |\ \nu(T_{\min}, x) > 0 \} 
\end{equation*}
where $T_{\min}$ is a current with minimal singularities in $\A$. 
\end{prop}
Precisely speaking, property (1) was proved only for the augmented base locus
in \cite[Section 5]{ELMNP09}. 
However, we shall give the proof for the non-K\"ahler locus
in the proof of Proposition \ref{ concave}. 

\subsection{Divisorial Zariski decompositions.}

In this subsection, we confirm the definition of the divisorial Zariski decomposition of a class.
The divisorial Zariski decomposition of a big divisor coincides with its $\sigma$-decomposition.
The divisorial Zariski decomposition is studied in \cite{Bou04} and 
the $\sigma$-decomposition is studied in \cite{Nak}.

Let $\A$ be a pseudo-effective class in $H^{1,1}(M,\RR)$.
Then the effective $\RR$-divisor $N$ is defined to be 
\begin{equation*}
 N:=\sum_{F:\mathrm{prime\ div}} \nu(\A , F) F. 
\end{equation*}
Here $\nu(\A, F)$ denotes the Lelong number along a prime divisor $F$ 
which is defined by $\inf_{x \in F} \nu(\A, x)$.
The class $\{N \}$ of $N$ is called the negative part of the 
divisorial Zariski decomposition of $\A$. 
The class $P$ defined by $P:= \A -\{ N \}$ is called the positive part.
Then the decomposition $\A = P + \{ N \}$ is said to be the divisorial Zariski decomposition of $\A$.
In general, the positive part $P$ is nef in codimension one
(that is, the codimension of its non-nef locus is strictly larger than one). 
We say that $\A$ admits a Zariski decomposition if the positive part $P$ is nef.
If $\A$ is the first Chern class of a big divisor, 
this definition coincides with that of the divisor 
(which was described in section 1).
For example, if $M$ is surface, any big class admits a Zariski decomposition 
(see \cite[section 4]{Bou04}).
By the construction of $N$, positive currents in $\A$ and positive currents in $P$ are identified  
by the correspondence $T \in \A  \longmapsto T-[N] \in P$.

\section{Restricted volumes and Zariski decompositions }

\subsection{The positive part and restricted volumes }
The main aim in this section is to prove Theorem \ref{ main}.
Throughout this section, 
let $D$ be a big divisor on a smooth projective variety $X$ of dimension $n$.  
Then we consider the divisorial Zariski decomposition $D = P+N$ of $D$.
We first establish Proposition \ref{ com-posi} for the proof of Theorem \ref{ main}. 
This proposition asserts that 
the restricted volume of $D$ can be computed with the positive part $P$.

\begin{prop}\label{ com-posi}
Let $W$ be an irreducible subvariety of dimension $d$ on $X$. 
Assume that $W$ 
is not contained in the augmented base locus $\mathbb{B}_{+}(D)$ of $D$.
Then the equality $\V{W} (D) = \V{W} (P)$ holds.
\end{prop}

\begin{rem}
In general, $P$ is an $\RR$-divisor. 
Then $\V{W} (P)$ can be defined by the limit of the restricted volumes of 
$\QQ$-divisors which converge to $P$ in the N\'eron-Severi space.
Thanks to the continuity of the restricted volume (see \cite[Theorem 5.2]{ELMNP09}), 
$\V{W} (P)$ does not depend on the choice
of $\QQ$-divisors which converge to $P$.
\end{rem}

\begin{proof}

Since $D$ is a big divisor, 
there is an effective $\QQ$-divisor which is $\QQ$-linearly 
equivalent to $D$.
Therefore we may assume that $D$ is effective. 
(Recall that the restricted volume has the homogeneity.)  
Moreover we may assume that the support of $D$ does not contain $W$, 
since $W$ is not contained in $\mathbb{B}_{+}(D)$. 
In particular, $W$ is not contained in the support of $N$ nor in that of $P$.

Since $D=P+N$ is a divisorial Zariski decomposition, 
there exists the natural isomorphism 
$H^{0}(X, \OX(\lfloor kP \rfloor)) \cong H^{0}(X, \OX(kD))$ induced by 
the section $e_{k}$ for a positive integer $k>0$, 
where $e_{k}$ is the standard section of the effective divisor $\lceil kN \rceil$ 
(see \cite[Theorem 5.5]{Bou04} or \cite{Nak}).
Then we consider the following commutative diagram: 
\[
\begin{CD} H^0(X,\OX (\lfloor kP \rfloor)) @>\cdot e_{k}>>H^0(X,\OX (kD))\\
@V{f}VV  @V{g}VV\\
H^0(W,\OW (\lfloor kP \rfloor)) @> \cdot e_{k} \vert_{W}>>H^0(W,\OW (kD)), 
\end{CD}
\]
where $f$ and $g$ are the restriction maps. 
The diagram induces the map 
$\mathrm{Im}(f) \xrightarrow[]{\cdot e_{k} \vert_{W}} \mathrm{Im}(g)$. 
This map is surjective since the horizontal map above in the diagram 
is an isomorphism. 
Now ${e_{k}} \vert_{W}$ is a nonzero section
since $W$ is not contained in the support of $N$. 
It implies that the map below 
in the diagram is injective. 
Thus, the map $\mathrm{Im}(f) \to \mathrm{Im}(g)$ is an isomorphism. 
It yields  
\begin{equation}
\V{W}(D) = \limsup _{k \to \infty}\dfrac{h^{0}(X|W,\OX (\lfloor kP \rfloor))}{k^{d}/d! }. \label{ eq}
\end{equation}

When $P$ is a $\QQ$-divisor, Proposition \ref{ com-posi} follows from this equality and 
the homogeneity of the restricted volume. 
However, we need the following argument  when $P$ is an $\RR$-divisor.

Let $P=\sum_{i} a_{i}D_{i}$ be the irreducible decomposition of $P$. 
Note that  $a_{i}$ is positive for any $i$ since $P$ is effective. 
We want to approximate the $\RR$-divisor $P$ with suitable $\QQ$-divisors. 
For this purpose, 
we define a $\QQ$-divisor $P_{\ell}$ by $P_{\ell} :={\ell}^{-1} {\lfloor \ell P \rfloor}$.
Then, from the definition of the round-down, 
we obtain $\lfloor \ell P \rfloor \leq \ell P\leq \lfloor \ell P \rfloor + F$ 
for any positive integer $\ell$, 
where $F$ is the effective divisor defined by $F:={ \sum_{i}D_{i} }$.
These inequalities imply that 
$P_{\ell}$ converges to $P$ in the N\'eron-Severi space.
For a sufficiently large $\ell$, $\mathbb{B}_{+}(P_{\ell})$ does not contain $W$  
(see Proposition \ref{ pro-bs} (1)). 
Therefore we have 
\begin{equation*}
 \V{W}(P) =\displaystyle{ \lim_{\ell \to \infty} \V{W}(P_{\ell}) }
\end{equation*} 
from the continuity of the restricted volume.

Now we prove the inequality $\V{W}(D) \geq  \V{W}(P_{\ell}) $ 
for any $\ell$ 
in order to show the inequality $\V{W}(D) \geq  \V{W}(P)$. 
By the homogeneity of the restricted volume and equality (\ref{ eq}), 
we obtain the following equalities:
\begin{align*}
\V{W}(D) = \limsup_{k \to \infty} \dfrac{h^{0}(X|W,\OX (\lfloor kP \rfloor))}{k^{d}/d! } 
&= \limsup_{k \to \infty} \dfrac{h^{0}(X|W,\OX (\lfloor \ell kP \rfloor))}{\ell^{d} k^{d}/d! }, \\
\V{W} (P_{\ell} ) &= \limsup _{k \to \infty} \dfrac {h^{0}(X|W,\OX (k \lfloor \ell P \rfloor))}{\ell^{d} k^{d}/d! }.
\end{align*}
Note that $\lfloor\ell kP \rfloor - k \lfloor \ell P \rfloor$ is an effective divisor and 
its support is contained in the support of $D$.
Since $W$ is not contained in the support of $D $, we have 
\begin{equation*}
\dfrac{ h^{0}(X|W ,\OX (\lfloor \ell kP \rfloor))} {\ell^{d} k^{d}/d!} \geq 
\dfrac{h^{0}(X|W,\OX (k \lfloor \ell P \rfloor))} {\ell^{d} k^{d}/d!}.
\end{equation*}
It implies the inequality $\V{W}(D) \geq  \V{W}(P_{\ell}) $ for any $\ell$. 
Therefore we obtain $\V{W}(D) \geq  \V{W}(P)$.

Finally we show the converse inequality $\V{W}(D) \leq  \V{W}(P)$. 
For this purpose,  
we shall estimate $\lfloor\ell kP \rfloor - k \lfloor \ell P \rfloor$ from above. 
By a simple computation, 
we obtain 
\begin{align*}
\lfloor\ell kP \rfloor - k \lfloor \ell P \rfloor &=
 \sum_{i}\big( \lfloor \ell k a_{i} \rfloor - k \lfloor \ell a_{i} \rfloor \big) D_{i} \\
&\leq  \sum_{i} \big(  k (\ell a_{i}  - \lfloor \ell a_{i} \rfloor ) \big) D_{i} \\
&\leq \sum_{i}  k  D_{i} = kF.
\end{align*}
Since the support of $F$ does not contain $W$, 
the inequality above yields 
\begin{align*}
\V{W}(D) & =\limsup_{k \to \infty} \dfrac{ h^{0}(X|W ,\OX ( \lfloor \ell kP \rfloor))} {\ell^{d} k^{d}/d!} \\
&\leq  \limsup_{k \to \infty} \dfrac{ h^{0}\big( X|W ,\OX \big( k (\lfloor \ell P \rfloor + F) \big) \big)} 
{\ell^{d} k^{d}/d!} \\
&= \dfrac{1}{\ell^{d}} \V{W} (\lfloor \ell P \rfloor + F).
\end{align*}
Now we have 
$\V{W} (\lfloor \ell P \rfloor + F) \ell^{-d} = \V{W} (P_{\ell} + {\ell}^{-1}{F} )$ 
from the homogeneity of the restricted volume. 
Further, $ P_{\ell}  + {\ell}^{-1}{F} $ converges to $P$ 
in the N\'eron-Severi space when $\ell$ tends to infinity. 
The continuity of the restricted volume implies  that 
$\V{W} (\lfloor \ell P \rfloor + F) \ell^{-d} $ 
converges to $\V{W}(P)$.
Hence we obtain the converse inequality $\V{W}(D) \leq  \V{W}(P)$. 
\end{proof}

\begin{cor}\label{ com-posico}
Let $W $ be an irreducible subvariety of dimension $d$ on $X$ 
Assume that $W$ is not contained in $\mathbb{B}_{+}(D)$.
If $D$ admits a Zariski decomposition (that is, the positive part $P$ of its 
divisorial Zariski decomposition is nef), 
then the equality $\V{W} (D) = (W\cdot P^{d})$ holds.
\end{cor}

\begin{proof}
By Proposition \ref{ com-posi}, we have $\V{W} (D) = \V{W}(P)$.
Since $P$ is nef, there exist ample $\QQ$-divisors $A_{k}$ which 
converge to $ P$ in the N\'eron-Severi space.
Since $A_{k}$ is ample, the restricted volume $\V{W}(A_k)$ 
of $A_{k}$ along $W$ is equal to the self-intersection number  $(W \cdot A_{k}^{d})$ on $W$. 
By the continuity of the restricted volume and the self-intersection number, 
we obtain 
\begin{align*}
\V{W} (D) &= \V{W}(P) \\
&= \lim_{k \to \infty} \V{W} (A_k) \\
&= \lim_{k \to \infty} (W \cdot A_{k}^{d}) = (W \cdot P^{d}).
\end{align*} \end{proof}

\subsection{Proof of Theorem \ref{ main}.}
This subsection is devoted to complete the proof of Theorem \ref{ main}. 
First, we shall see that condition (1) implies condition (2). 
We assume that $D$ admits a Zariski decomposition $D=P+N$ 
(that is, the positive part $P$ of its divisorial Zariski decomposition is nef). 
Take subvarieties $V$ and $V^{'}$ on $X$ such that $V \equiv V^{'}$ and $V, V^{'} \not\subseteq \mathbb{B}_{+}(D)$. 
Then the restricted volumes of $D$ can be computed 
by the self-intersection number of the positive part $P$ from Corollary \ref{ com-posico}.
That is, $\vol(D) =(V\cdot P^{d}) $ and $\vo(D) =(V^{'}\cdot P^{d}) $ hold.
Since $V$ and $V^{'}$ are numerically equivalent, 
$(V\cdot P^{d})$ coincides with $(V^{'}\cdot P^{d})$.
Hence the equality $\vol(D) = \vo (D)$ holds.

We shall show that condition (3) implies condition (1). 
Let $D=P+N$ be a divisorial Zariski decomposition of a big divisor $D$. 
We assume that $P$ is not nef for a contradiction. 
Since $P$ is not nef, 
the restricted base locus $\mathbb{B}_{-}(P)$ of $P$ is not empty. 
From this condition, we want to construct curves $C$, $C^{'}$ 
such that the restricted volume $\V{C} (P)$ along $C$ is different from 
the restricted volume $\V{C^{'}} (P)$ along $C^{'}$.

For a construction of such curves, 
we take a very ample divisor $A$ on $X$ 
and a point $x_{0}$ in $\mathbb{B}_{-}(P)$. 
Then there are smooth curves $C$ and $C^{'}$ 
with the following \vspace{0.1cm}properties:\\
\ \ \ (1)\ \ \ $C$ and $C^{'}$ are not contained in the augmented base locus $\mathbb{B}_{+}(D)$. \\
\ \ \ (2)\ \ \ $C$ passes through  $x_{0} \in \mathbb{B}_{-}(P)$.\\
\ \ \ (3)\ \ \ $C^{'}$ does not intersect with the restricted base locus $\mathbb{B}_{-}(P)$.\\
\ \ \ (4)\ \ \ $C$ and $C^{'}$ are complete intersections of members of the complete linear system of $A$. \vspace{0.1cm}

We can easily see that there exist such curves: 
A general member of $ |A|_{x_{0}}$ is irreducible and smooth,
where $ |A|_{x_{0}}$ is the linear system passing through $x_{0}$ in 
the complete linear system $ |A|$ of $A$ 
(see \cite[Theorem 2.5]{Zha09}). 
Then by taking a complete intersection of general members of $ |A|_{x_{0}}$, 
we can take a curve $C$ with properties (1), (2), (4). 
Now we construct a curve $C^{'}$ with properties (1), (3), (4). 
By the construction of the divisorial Zariski decomposition, 
the restricted base locus $\mathbb{B}_{-}(P)$ of the positive part $P$ is 
the countable union of subvarieties of codimension $\geq 2$. 
Thus the codimension of the intersection of $\mathbb{B}_{-}(P)$ and $H$ 
is greater than 
or equal to $3$ for a \lq\lq very" general member $H$ of $|A|$. 
It implies that a curve which is 
a complete intersection of very general members of $|A|$ 
does not intersect with $\mathbb{B}_{-}(P)$.

Now $C$ and $C^{'}$ are numerically equivalent 
since $C$ and $C^{'}$ are complete intersections 
of members of the same complete linear system. 
Thus, it follows 
the equality $\V{C}(D) = \V{C^{'}}(D) $ from condition (3) in Theorem \ref{ main}. 
By Proposition \ref{ com-posi}, we have the equality $\V{C}(P) = \V{C^{'}}(P) $. 
It is sufficient for a contradiction to prove the following lemma.  
In fact, $(C \cdot P)$ is equal to $(C^{'} \cdot P)$
since $C$ and $C^{'}$ are numerically equivalent.  
Therefore the following lemma implies $\V{C}(D) < \ \V{C^{'}}(D) $. 
It is a contradiction.

\begin{lemm} \label{ estimate}
In the situation above, the followings hold. \\
\hspace{1cm} $\mathrm{(A)}$\ \ \ $\V{C^{'}}(P) = (C^{'} \cdot P)$, 
\hspace{1cm} $\mathrm{(B)}$\ \ \ $ \V{C}(P) < (C\cdot P)$.
\end{lemm}

\begin{proof}
First we show equality (A). 
In general, $P$ is an $\RR$-divisor. 
Thus we should approximate $P $ with $\QQ$-divisors. 
We take ample $\RR$-divisors $A_{k}$ with the following properties: \\
\ \ \ (i)\ \ \ $P + A_{k}$ is a $\QQ$-divisor for a positive integer 
$k>0$. \\
\ \ \ (ii)\ \ \ $P + A_{k}$ converges to $P$ in the N\'eron-Severi space.

Since $A_{k}$ is ample, $\mathbb{B}_{+}(P+A_{k}) \subseteq  \mathbb{B}_{+}(P)$ for 
any $k$. 
Thus, it follows that $C^{'}$ is not contained in $\mathbb{B}_{+}(P+A_{k})$ 
from property (1). 
We take  a positive integer $a_{k}$ such that $a_{k} (P+A_{k})$ is a $\ZZ$-divisor. 
Then the homogeneity and the description of the restricted volume 
with the asymptotic multiplier ideal sheaf 
(which was proved in \cite[Theorem 2.13]{ELMNP09}) yields
\begin{equation*}
\V{C^{'}}(P+A_{k}) = \limsup_{\ell \to \infty} 
\dfrac{1}{{\ell a_{k}}}
h^{0} \big( C^{'}, \mathcal{O}_{C^{'}} \big( \ell a_{k} (P+A_{k}) \big) \otimes \J 
\big(\Vert \ell a_{k} (P+A_{k}) \Vert\big)\big\vert_{C^{'}}  \big). 
\end{equation*}

Here $\J (\Vert L \Vert )$ denotes the asymptotic multiplier ideal sheaf associated 
to a divisor $L$ 
(see \cite{DEL00} for the definition). 
Further $\mathcal{I} \vert_{V}$ denotes the ideal $\mathcal{I} \cdot \mathcal{O}_{V}$
for a ideal $\mathcal{I}$ and a subvariety $V$ on $X$.

We shall investigate the asymptotic multiplier ideal sheaf 
$\J (\Vert \ell a_{k} (P+A_{k}) \Vert )$ along ${C^{'}}$. 
Let $T_{\mathrm{min},k}$ be a current with minimal singularities in 
the first Chern class of $a_{k} (P+A_{k})$. 
Then the restricted base locus of $a_{k} (P+A_{k})$ is equals to the set 
$\big\{x \in X \ \big| \ \nu(T_{\mathrm{min},k}, x) >0 \big\}$ 
by Proposition \ref{ pro-bs} (3). 
On the other hand, 
the restricted base locus of $P+A_{k}$ is contained in 
the restricted base locus of $P$, since $A_{k}$ is an ample divisor. 
Further,  $C^{'}$ does not intersect with the restricted base locus of $P$ 
from property (3). 
Therefore the Lelong number of $T_{\mathrm{min},k}$ at $x \in C^{'}$ 
is zero. 
Thus we have
$\J(\ell T_{\mathrm{min},k}) \vert_{C^{'}} = \mathcal{O}_{C^{'}}$ 
for every $\ell >0$ by Skoda's Lemma. 
Thus, from Theorem \ref{ min} (which is proved in section 4.1), 
we have 
\begin{equation*}
 \V{C^{'}}(P+A_{k}) = \limsup_{\ell \to \infty} 
\dfrac{h^{0} \big( C^{'}, \mathcal{O}_{C^{'}} ( \ell a_{k} (P+A_{k}) \big)}{\ell a_{k}} .
\end{equation*}

Since $C^{'}$ is not contained in $\mathbb{B}_{+}(P)$, 
$(P+A_{k})$ is an ample divisor on $C^{'}$. 
By the Riemann-Roch formula, 
we obtain $\V{C^{'}} (P+A_{k}) = \big((P+A_{k}) \cdot C^{'} \big)$.
It follows $\V{C^{'}}(P) = (P \cdot C^{'} )$ from the continuity of 
the restricted volume.

Finally we show inequality (B). 
Consider the following commutative diagram: 
\[
\begin{CD} H^0\big( X,\OX (\lfloor kP \rfloor)\otimes \J(\Vert\lfloor kP \rfloor \Vert) \big) 
@>>>H^0 \big( C,\mathcal{O}_{C} (\lfloor kP \rfloor)\otimes \J(\Vert\lfloor kP \rfloor \Vert)\vert_{C} \big) \\
@VVV  @VVV\\
H^0\big( X,\OX (\lfloor kP \rfloor) \big) @>>>H^0 \big( C,\mathcal{O}_{C} (\lfloor kP \rfloor) \big).
\end{CD}
\]
The vertical map on the left hand is an isomorphism (see \cite[Theorem 1.8]{DEL00}).
Thus the vertical map on the right hand is surjective onto the image of the horizontal map. 
It yields 
\begin{equation*}
 \limsup_{k \to \infty} \dfrac{h^0\big(X |{C} ,\OX (\lfloor kP \rfloor) \big)}{k} \leq  
 \limsup_{k \to \infty} \dfrac{h^0 
 \big( C,\mathcal{O}_{C} (\lfloor kP \rfloor)\otimes 
 \J(\Vert\lfloor kP \rfloor \Vert)\vert_{C} \big) }{k}.
\end{equation*}
We have already proved that the left hand 
is equals to $\V{C}(P)$ in the proof of 
Proposition \ref{ com-posi}. 
For the proof of inequality (B), 
we need to estimate the right hand from above. 
For this purpose, we shall investigate the asymptotic multiplier ideal sheaf 
$\J(\Vert\lfloor kP \rfloor \Vert) $ along $C$. 

Take a positive current $S_{k}$ in the first Chern class of 
$\lfloor kP \rfloor$ such that $\J(S_{k}) = \J(\Vert\lfloor kP \rfloor \Vert)$. 
Let $P=\sum{a_{i} D_{i}}$ be an irreducible decomposition of $P$ and let 
$F$ the divisor which is defined by $F:=\sum_{i} D_{i}$. 
Then $kP - \lfloor kP \rfloor \leq F$ for any positive integer $k$. 
Thus we obtain 
\begin{equation*}
\nu(kT_{\mathrm{min}}, x) \leq \nu(S_{k},x) + \nu([F], x)
\end{equation*}
by the definition of a current with minimal singularities. 
Here 
$[F]$ denotes the positive current defined by the effective divisor $F$ 
and $T_{\mathrm{min}}$ denotes a current with minimal singular in $c_{1}(P)$. 
Since $\nu(T_{\mathrm{min}},x_{0}) $ is positive by property (2), 
we can take a positive rational number $p/q$ 
which is smaller than $\nu(T_{\mathrm{min}},x_{0})$. 
For simplicity, we put $c:=\nu([F],x_{0})$. 
Then we have $k p - c < \nu(S_{kq},x_{0})$. 
Skoda's Lemma implies 
$\J(\Vert\lfloor kP \rfloor \parallel)_{x_{0}} \subseteq \frak{m}_{X, x_{0}}^{{kp -c-n+1}}$, 
where $\frak{m}_{X, x_{0}}$ is the maximal ideal in 
$\mathcal{O}_{X, x_{0}}$.
Thus we obtain 
\begin{align*}
\limsup_{k \to \infty} \dfrac{h^0 \big( C,\mathcal{O}_{C} (\lfloor kP \rfloor)\otimes \J(\Vert\lfloor kP 
\rfloor \Vert)\vert_{C} \big) }{k}   
&\leq  \limsup_{k \to \infty} \dfrac{h^0 \big( C,\mathcal{O}_{C} (\lfloor kqP \rfloor)\otimes
 {\frak{m}^{{kp -c-n+1}}_{X, x_{0}}} \vert_{C} \big) }{kq}  \\
 &\leq  \limsup_{k \to \infty} \dfrac{h^0 \big( C,\mathcal{O}_{C} (\lfloor kqP \rfloor)\otimes
{\frak{m}^{{kp -c-n+1}}_{C, x_{0}}}\big) }{kq}  \\
& = \limsup_{k \to \infty} \dfrac{h^0 \big( C,\mathcal{O}_{C} (\lfloor kqP \rfloor - (kp -c-n+1)[x_{0}] )\big) } {kq} ,
\end{align*}
where $[x_{0}]$ is the divisor on $C$ defined by $x_{0}$. 
Now $(\lfloor qkP \rfloor - (kp-c-n+1)[x_{0}] )$ is nef on $C$. 
Thus the dimension of its first cohomology group converges to zero 
when $k$ tends to infinity . 
By using the Riemann-Roch formula again, we obtain 
\begin{align*} \V{C}(P)
&\leq  \limsup_{k \to \infty} \dfrac{h^0 \big( C,\mathcal{O}_{C} 
(\lfloor kqP \rfloor - (kp-c-n+1)[x_{0}] )\big) } {kq} \\
&= \limsup_{k \to \infty} \dfrac
{\big( C \cdot (\lfloor kqP \rfloor - (kp-c-n+1)[x_{0}]) \big)}{kq} \\
&\leq (C\cdot P) -\dfrac{p}{q} < (C \cdot P).
\end{align*}
\end{proof}

In the proof of Lemma \ref{ estimate}, we have already proved the following corollary.
\begin{cor}\label{ estimateco}
Let $C$ be a smooth curve on $X$. 
Assume that $C$ is not contained in $\mathbb{B}_{+}(D)$.
Then we have 
\begin{equation*}
\V{C}(D) \leq (C\cdot D) - \sum_{x \in C \cap \mathbb{B}_{-}(D)} \nu(T_{\mathrm{min}},x).
\end{equation*}
\end{cor}

\section{The analytic description of the restricted volume \\with positive curvature currents.}
\subsection{Proof of Theorem \ref{ main2}}
The main aim of this subsection is to prove Theorem \ref{ main2}.
Before the proof of Theorem \ref{ main2}, we need to show that the integral 
in Theorem \ref{ main2} is always finite.

\begin{prop}\label{ finite}
Let $W$ be an irreducible analytic subset 
of dimension $d$ on a compact K\"ahler manifold $M$ 
and $T$ a positive $d$-closed $(1,1)$-current on $M$.
Assume that the polar set of a potential function of $T$ does not contain $W$.
Then the integral $\int_{W_{\reg}}({T\vert_{W_{\reg} } })_{\ac}^d$ is finite.
\end{prop} 

\begin{proof}
In \cite[Lemma 2.11]{Bou02}, it has been proved that 
$\int_{W}S_{\ac}^d$ is finite for a positive $d$-closed current $S$ on $W$ 
when $W$ is non-singular. 
Since $T$ is a positive $d$-closed current on $M$, 
the restriction $T\vert_{W_{\reg}}$ is also a positive $d$-closed current. 
Thus, Proposition \ref{ finite} holds when $W$ is non-singular. 
It is enough to consider the case when $W$ has singularities. 
Then we take an embedded resolution $\mu :\widetilde {W} \subseteq  \widetilde {M}  \longrightarrow W \subseteq M$ of $W \subseteq M$. 
That is, $\mu$ is a modification from a compact complex manifold 
$\widetilde{M}$ to $M$ and 
its restriction to $W$ gives a resolution of singularities of $W$. 
Then the following lemma assures that Proposition \ref{ finite} holds 
even if $W$ has singularities.
In fact, we have 
$$\int_{\widetilde{W} } { \big(  (\mu^{*} T)\vert_{\widetilde{W}} \big)_{\ac}^{d} } 
= \int_{W_{\reg}}{({T\vert_{W_{\reg}}})_{\ac}^d}$$ 
by this lemma. 
The left hand is finite since $\widetilde{W}$ is non-singular. 
Thus the right hand is also finite. 
\end{proof}

\begin{lemm}\label{ bir}
Let $\mu :\widetilde {W} \subseteq  \widetilde {M}  \longrightarrow W \subseteq M$ be 
an embedded resolution of $W \subseteq M$. 
In the assumption of Proposition \ref{ finite}, we have 

\[ \ \ \ \ \int_{W_{\reg}}{({T\vert_{W_{\reg}}})_{\ac}^d} = 
\int_{\widetilde{W} } { \big(  (\mu^{*} T)\vert_{\widetilde{W}} \big)_{\ac}^{d} }.\]
\end{lemm}

\begin{proof}
The map $\widetilde {W}\xrightarrow{\ \ \mu \ \ } W$ is a modification. 
Therefore $(\mu ^{*}T) \vert _{\widetilde{W}}$ is identified with $T\vert_{W_{\reg}}$  by $\mu$ on some Zariski open set. 
Now 
$\big( (\mu ^{*}T) \vert _{\widetilde{W}} \big)_{\ac}$ and $\big( T\vert_{W_{\reg}} \big)_{\ac}$
are $(1,1)$-forms with $L^{1}$-functions as coefficients. 
Since a Zariski closed set is of measure zero with respect to the Lebesgue measure,  
we obtain $\int_{W_{\reg}}{({T\vert_{W_{\reg}}})_{\ac}^d} = 
\int_{\widetilde{W} } { \big(  (\mu^{*} T)\vert_{\widetilde{W}} \big)_{\ac}^{d} }$ .

\end{proof}
The rest of this subsection is devoted to prove Theorem \ref{ main2}. 
\hspace{-6mm}\\
\textit{Proof of Theorem \ref{ main2}.})\\
{\bf{ (Step1)} } In this step, we prove the inequality $\geq$ in Theorem \ref{ main2}
by using the singular holomorphic Morse inequalities (see [Bon98]) and Proposition \ref{ min}.
Proposition \ref{ min} is proved at the end of this subsection.
 Let $T$ be a positive $d$-closed $(1,1)$-current 
with analytic singularities in the first Chern class $c_1(D)$ 
whose singular locus does not contain $V$.

First, we consider the case when $V$ is non-singular.
Then we obtain the following inequality: 
\begin{align*}
\vol (D) &= \limsup_{k \to \infty} \frac{h^{0} \bigl(V,\mathcal{O}_{V} (kD) \otimes \mathcal{I} (kT_{\mathrm{min}})\vert_{V} \bigr)}{{ k^d}/{d!} } \\
& \geq \limsup_{k \to \infty} \frac{h^{0} \bigl(V,\mathcal{O}_{V} (kD) \otimes \mathcal{I} (kT)\vert_{V} \bigr)}{{ k^d}/{d!} } \\
& \geq \limsup_{k \to \infty} \frac{h^{0} \bigl(V,\mathcal{O}_{V} (kD) \otimes \mathcal{I} (kT\vert_{V})\bigr)}{{ k^d}/{d!} }.
\end{align*}

Here $T_{\min}$ denotes a current with minimal singularities in $c_{1}(D)$.
The first equality follows from Proposition \ref{ min} and the second inequality follows from the restriction formula (see \cite[Corollary 1.3]{DEL00} for the restriction formula).
By using the singular holomorphic Morse inequality, we have 
\begin{align*}
\vol (D) &\geq \limsup_{k \to \infty} \frac{h^{0} \bigl(V,\mathcal{O}_{V} (kD) \otimes \mathcal{I} (kT\vert_{V})\bigr)}{{ k^d}/{d!} }\\
& \geq \int_{V}{{(T\vert_{V})}_{\ac}^d}.
\end{align*}
Therefore the inequality $\geq$ in Theorem \ref{ main2} holds  
when $V$ is non-singular.

Now we consider the case when $V$ has singularities.
Then we take an embedded resolution 
$\mu: \widetilde V \subseteq \widetilde X \longrightarrow V\subseteq X$.
The augmented base locus of the pull back $\mu ^{*}D$ of $D$ 
does not contain $\widetilde V$, 
since  $\mu : \widetilde V \longrightarrow V$ is a modification.
By applying the singular holomorphic Morse inequality and 
restriction formula to $\mu ^{*}D$, $\widetilde V$ and $\mu ^{*}T$ again, 
we obtain
\begin{equation*} 
\voll (\mu ^{*}D) \geq \int_{\widetilde V} { ((\mu ^{*}T)\mid _{\widetilde V})_{\ac}^{d}}.
\end{equation*}
By Lemma \ref{ bir},  we obtain 
$\int_{V} \T=\int_{\widetilde{V}} \Tm$. 
On the other hand, we have $\vol (D)=\voll (\mu ^{*}D)$ from \cite[Lemma 6.7]{ELMNP09}. 
Thus the inequality $\geq$ in Theorem \ref{ main2} holds even if $V$ has singularities.
\hspace{-6mm}\\

{\bf(Step2)}\ \ In this step, 
we shall show the converse inequality $\leq$ 
by applying Fujita's approximation theorem 
for the restricted volume (which is proved in \cite[Proposition 2.11]{ELMNP09}). 
By applying \cite[Proposition 2.11]{ELMNP09}, 
for an arbitrary number $\e >0$,  
we can find a modification $\pi_{\e}:X_{\e} \longrightarrow X$
and the expression ${\pi_{\e}}^{*} D=A_{\e}+E_{\e}$ 
such that $({A_{\e}}^{d}\cdot V_{\e}) \geq \vol(D) - \e$. 
Here $A_{\e}$ is an ample $\QQ$-divisor and $E_{\e}$ is an effective $\QQ$-divisor 
whose support does not contain 
the strict transformation $V_{\e}$ of $V$.

Let $\omega _{\e}$ be 
a K\"ahler form on $X_{\e}$ in the first Chern class of $A_{\e}$.
Since the support of $E_{\e}$ does not contain $V_{\e}$, 
we may restrict  $[E_{\e}]$ to $V_{\e}$, 
where $[E_{\e}]$ denotes the current defined by the effective divisor $E_{\e}$. 
Then we obtain
\begin{align*}
({A_{\e}}^{d} \cdot {V}_{\e}) &= \int_{V_{\e}}{({\omega _{\e}}\vert_{V_{\e}})^{d}} \\
&= \int_{V_{\e}}{ \big( { (\omega _{\e} + [E_{\e}])\vert_{V_{\e} } } \big) _{\ac}^{d}} \\
&= \int_{V_{\reg}} { \big\{ \big ({\pi _{\e}}_{*}(\omega _{\e} + [E_{\e}])\big)\vert_{V_{\reg} } \big\} _{\ac}^{d} } . \\
\end{align*}
The third equality follows from the same argument as the proof of Lemma \ref{ bir}.

Since $\pi_{\e}$ is a modification, 
its push-forward ${\pi _{\e}}_{*}(\omega _{\e} + [E_{\e}])$ is 
a positive current in the Chern class $c_{1}(D)$. 
However the push-forward may not have analytic singularities.
For the proof, we need to 
approximate the push-forward by positive currents with analytic singularities. 
When we approximate a given current by Theorem \ref{ app-bou02}, 
the approximation sequence may lose positivity. 
Now $(\omega _{\e} + [E_{\e}])$ is a K\"ahler current but 
the push-forward may not be a K\"ahler current.
For that reason, we need to consider a K\"ahler current 
before we apply Theorem \ref{ app-bou02}.

For simplicity, we put $T_{\e} := {\pi _{\e}}_{*}(\omega _{\e} + [E_{\e}])$. 
Since $V$ is not contained in the augmented base locus $ {\mathbb{B}}_{+}(D)$, 
there is a K\"ahler current $S$ in $c_{1}(D)$ 
with analytic singularities whose singular locus does not contain $V$.  
By Fatou's Lemma, we obtain

\begin{align*}
\vol(D) - \e & \leq  ({A_{\e}}^{d}\cdot V_{\e}) \\
&\leq \int_{\Vr} \liminf _{\del \to 0} \big\{ (1-\del)({T_{\e}}\vert_{\Vr}) + \del (S\vert_{V_{\reg}})  \big\}_{\ac}^{d} \\
&\leq \liminf _{\del \to 0} \int_{\Vr} \big\{ (1-\del)({T_{\e}}\vert_{\Vr}) + \del (S\vert_{V_{\reg}})  \big\}_{\ac}^{d} .
\end{align*}
Hence there exists a sufficiently small number $\del_{0} >0$ with the following 
inequality: 
\begin{equation*}
\vol(D) - 2\e \leq \int_{\Vr} \{ (1-\del_{0})({T_{\e}}\vert_{\Vr}) + \del_{0} (S\vert_{V})  \}_{\ac}^{d}.
\end{equation*}
Note that $(1-\del_{0})T_{\e} + \del_{0} S$ is a K\"ahler current in $c_{1}(D)$.
By applying the approximation theorem (Theorem \ref{ app-bou02} and Corollary \ref{ app-bou02co}) to $(1-\del_{0})T_{\e} + \del_{0} S$, 
we can find positive currents $ U_{k} $ in $c_1(D)$ with the following properties.\vspace{1mm}
\\
\ \ \ (1)\ \ \ $U_{k}$ has analytic singularities for every integer $k$. \\
\ \ \ (2)\ \ \ $({U_{k}}\vert_{\Vr})_{\ac} \longrightarrow 
\big\{ (1-\del_{0})T_{\e} \vert_{\Vr}+ \del_{0} S\vert_{\Vr} \big\}_{\ac}$\ \ \ 
 almost everywhere on $V_{\reg}$  \\
\ \ \ (3)\ \ \ $U_{k}$ is a positive current for a sufficiently large $k$. 
\vspace{1mm}

Fatou's Lemma and Property (2) imply  
\begin{align*}
\vol(D) - 2\e &\leq  \int_{\Vr} \big\{ (1-\del_{0})({T_{\e}}\vert_{\Vr})_{\ac} + \del_{0} (S\vert_{V_{\reg}})_{\ac}  \big\}^{d} \\
&= \int_{\Vr} \liminf_{k \to \infty}({U_{k}}\vert_{\Vr})_{\ac}^{d}\\
&\leq   \liminf_{k \to \infty} \int_{\Vr} ({U_{k}}\vert_{\Vr})_{\ac}^{d}.\\
\end{align*}
Therefore we have $$\vol(D)-3\e\leq \int_{\Vr} ({U_{k_{0}}}\vert_{\Vr})_{\ac}^{d}$$ for a sufficiently large $k_{0}$.
Now $\e$ is an arbitrary positive number and $U_{k_{0}}$ is a positive current with analytic singularities 
in the Chern class $c_1(D)$.
It completes the proof of Theorem \ref{ main2}. \begin{flushright}
$\square$
\end{flushright}
At the end of this subsection, we prove the following proposition, which is a variation of \cite[Theorem 2.13]{ELMNP09}.

\begin{prop}\label{ min}
Let $V$ be an irreducible subvariety of dimension $d$ on $X$. 
Assume that $V$ is not contained in $ {\mathbb{B}}_{+}(D)$. 
Then the following equality holds. 

\begin{equation*}
\vol (D) = \limsup_{k \to \infty} \frac{h^{0}\bigl(V,\mathcal{O}_{V} (kD) \otimes \mathcal{I} 
(kT_{\mathrm{min}})|_{V} \bigr)}{{ k^d}/{d!} } ,
\end{equation*}
where $T_{\mathrm{min}}$ is a current with minimal singularities in $c_1(D)$.
\end{prop}

\begin{proof}
By \cite[Theorem 2.13]{ELMNP09}, we know 
\begin{equation*}
\vol (D) = \limsup_{k \to \infty} \frac{h^{0}\bigl(V,\mathcal{O} 
(kD) \otimes \mathcal{J} ( \Vert kD \Vert)|_{V} \bigr)}{{ k^d}/{d!} } .
\end{equation*}
In order to prove Proposition \ref{ min}, we should compare the multiplier ideal sheaf 
$\mathcal{I}(kT_{\mathrm{min}})$ with the asymptotic multiplier ideal sheaf $\mathcal{J}(\Vert kD\Vert)$.
By the definition of a current with minimal singularities, we have  
$\mathcal{J}(\Vert kD\Vert)  \subseteq \mathcal{I}(kT _{\mathrm{min}}) $ 
for all positive integer $k$.
Hence it is sufficient to prove the inequality $ \geq $ in Proposition \ref{ min}.
For this purpose, we establish the following lemma.
\begin{lemm}\label{ compare}
Let $D$ be a big divisor on a smooth projective variety $X$. 
There is an effective divisor $E$ (independent of $k$) with the following properties:

\ \ \ \ $\mathrm{(i)}$\ \  \ \ \ The support of $E$ does not contain V.

\ \ \ \ $\mathrm{(ii)}$\ \ \ \ \ $\mathcal{I}(kT_{\mathrm{min}})\cdot \mathcal{O}_{X}(-E) 
\subseteq \mathcal{J}(\Vert kD\Vert )$
 for a sufficiently large $k$.
\end{lemm}

\begin{proof}
This proof is essentially based on the argument in \cite[Theorem 1.11]{DEL00}.
Fix a very ample divisor $A$ on $X$.
For an arbitrary point $x \in X$, there exists a zero-dimensional complete intersection $P_{x}$ 
of the complete linear system $|A|$ containing $x$. 
The Ohsawa-Takegoshi-Manivel $L^2$-extension theorem asserts that, 
there exists an ample divisor $B$ (which depends only on $A$) such that 
for any divisor $F$ and a singular hermitian metric $h$ on $F$ 
with the positive curvature current $T_{h}$,   
the following restriction map is surjective (see \cite{OT87}, \cite{Man93}):

\begin{equation*}
H^{0}\bigl(X,\mathcal{O}_{X} (F+B) \otimes \mathcal{I}(T_{h} )\bigl) 
\longrightarrow H^{0}\bigl(P_{x},\mathcal{O}_{P_{x}} (F+B)\otimes \mathcal{I}(T_{h} \vert_{P_{x}})\bigl )
\end{equation*}
Moreover, the Ohsawa-Takegoshi-Manivel $L^2$-extension theorem claims that,  
for a section on $P_{x}$, the extended section satisfies an $L^{2}$-estimate 
with a constant which is independent of $F$.
Further the $L^{2}$-estimate depends only on $A$.

Since $D$ is big and $V$ is not contained in the augmented base locus of $D$, 
we can take $E \in \big| k_{0}D - B \big|$ with property (i) by choosing a sufficiently large $k_{0}$.
We apply the Ohsawa-Takegoshi-Manivel $L^{2}$-extension theorem to 
$F_{k}:=(k-k_0)D + E$ equipped with a singular hermitian metric $h_{\mathrm{min}}^{\otimes k-k_{0}}\otimes h_{E}$.
Here $h_{\mathrm{min}}$ denotes a singular hermitian metric with minimal singularities 
and $h_{E}$ denotes a singular hermitian metric defined by the standard section of the effective divisor $E$.
Then for a sufficiently large $k$ and a given point $x\in X$, 
we obtain the global section $s_{x}$ of  
$F_{k}+B \sim kD$ with the following estimates:
 
\begin{equation*}
 \int_{X}{ |  s_{x} {|^{2}}_{{h_{\mathrm{min}}^{\otimes k-k_{0}}\otimes h_{E}\otimes h_{B} } } \omega^{n} \leq C\ \ \ \ \ \ and\ \ \ \ \  
| s_{x}(x) |^{2}}_{h_{\mathrm{min}}^{\otimes k-k_{0}}\otimes h_{E}\otimes h_{B} }  =1,
 \end{equation*}
where $C$ is a constant depending only $A$ and $h_{B}$ is a smooth hermitian metric on $B$ with the positive curvature. 
Here $\omega$ is a K\"ahler form on $X$. 
From the second equality, we obtain 
\begin{equation*}
|s_{x}(x)|^{2}e^{-2(k-k_{0}) \varphi _{\mathrm{min}} -2\varphi _{E}-2\varphi _{B} }=1,
\end{equation*} 
where $\varphi _{\mathrm{min}}$, $\varphi _{E}$, $\varphi _{B}$ is the weight of the hermitian metric 
$h_{\mathrm{min}}$,$h_{E}$,$h_{B}$ respectively.
Since $\varphi _{B}$ is a smooth function and $X$ is compact, there is a 
constant $C^{'}$ such that 
 \begin{equation*}
 \varphi _{\mathrm{min}}+\frac{1}{k-k_0}\varphi _{E} \leq \frac{1}{k-k_{0}} \log|s_{x}(x)| +C^{'}.
 \end{equation*}

The evaluation map 
$H^{0}\bigl(X,\mathcal{O}_{X} (kD))\longrightarrow \CC$ is 
a bounded operator on the Hilbert space 
$H^{0}\bigl(X,\mathcal{O}_{X} (kD))$ with the $L^{2}$-norm. 
Moreover the operation norm is equal to the Bergman kernel
 $${\sum_{j=1}^{N_{k}} 
 {|f_{j}}(x)|_{h_{\mathrm{min}}^{\otimes k-k_{0}}\otimes h_{E}\otimes h_{B} }  } 
 $$ 
where $\{f_{j}\}_{j=1}^{N_{k}}$ 
is an orthonormal basis of $H^{0}\bigl(X,\mathcal{O}_{X} (kD))$.
Therefore there is a constant $C^{''}$ such that 
\begin{equation*}
 \log|s_{x}(x)| \leq 
 \log\sum_{j=1}^{N_{k}}{|f_{j}}| +C^{''}. 
\end{equation*}
These inequalities imply that the function $\frac{1}{k-k_{0}} \log\sum_{j}{|f_{j}}|$ has 
less singularities than $\varphi _{\mathrm{min}}+\frac{1}{k-k_0}\varphi _{E} $.
By the definition of the asymptotic multiplier ideal sheaf,  we obtain property (ii). 
\end{proof}

We shall complete the proof of Proposition \ref{ min} by using Lemma \ref{ compare}.
From property (i) in the previous lemma, we may consider the following short exact sequence:
\begin{equation*}
0 \longrightarrow  \mathcal{O}_{V}(kD-E)     \longrightarrow    \mathcal{O}_{V}(kD)  \longrightarrow
     \mathcal{O}_{V\cap E}(kD)     \longrightarrow	0 .
\end{equation*}
Since the dimension of the intersection $V\cap E$ is smaller than $\dim{V}=d$, 
we have
\begin{equation*}
\limsup_{k \to \infty} \dfrac{h^{0} 
\big(V \cap E,\mathcal{O}_{V \cap E} (kD) \big)}{k^{d}/d!} =0.
\end{equation*}
Hence we obtain 
\begin{equation*}
\limsup_{k \to \infty} \frac{h^{0} \bigl(V,\mathcal{O}_{V} (kD) 
\otimes \mathcal{I} (kT_{\mathrm{min}})|_{V} \bigr)}{{ k^d}/{d!} } \leq
  \limsup_{k \to \infty} \frac{h^{0} \bigl(V,\mathcal{O}_{V} (kD-E) 
\otimes \mathcal{I} (kT_{\mathrm{min}})|_{V}  \bigr)}{{ k^d}/{d!} } .
\end{equation*}
On the other hand, $E$ satisfies property (ii) in Lemma \ref{ compare}. 
It implies 
 \begin{equation*}
\limsup_{k \to \infty} \frac{h^{0} \bigl(V,\mathcal{O}_{V} (kD-E) \otimes \mathcal{I}
 (kT_{\mathrm{min}})|_{V}  \bigr)}{{ k^d}/{d!} } \leq  
\limsup_{k \to \infty} \frac{h^{0} \bigl(V,\mathcal{O}_{V} (kD) \otimes \mathcal{J}
(\Vert kD \Vert )|_{V} \bigr)}{{ k^d}/{d!} }.
\end{equation*}
These inequalities assert 
\begin{equation*}
 \limsup_{k \to \infty} \frac{h^{0} \bigl(V,\mathcal{O}_{V} (kD) 
\otimes \mathcal{I} (kT_{\mathrm{min}})|_{V} \bigr)}{{ k^d}/{d!} }  
\leq \vol (D) .
\end{equation*}
\end{proof}

\subsection{Properties of the restricted volume.}
Theorem \ref{ main2} enables us to define 
the restricted volume for a big class on 
a compact K\"ahler manifold (see Definition \ref{ def}). 
In this subsection, we study the properties of the restricted volume of a class on a 
compact K\"ahler manifold.
Throughout this subsection, 
we denote by $M$ a compact K\"ahler manifold and by $W$ an irreducible analytic subset 
on $M$ of dimension $d$ and 
by $\A$ a big class in $H^{1,1}(M,\RR)$.

\begin{prop}\label{ nef}
Assume that $\A$ is a nef class and $W$ is not contained in the non-K\"ahler locus $E_{nK}(\A)$ of $\A$.
Then the restricted volume $\Vm{W}(\A)$ is equal to 
the self-intersection number  $(\A^{d}\cdot W)$ on $W$.
\end{prop}

\begin{proof}
When $W$ is non-singular, this proposition is proved by using the same argument as \cite[Theorem 4.1]{Bou02}. 
By using Lemma \ref{ bir}, we can give  
the proof even if $W$ has singularities. 
 \end{proof}

The following proposition is the generalization of 
Proposition \ref{ com-posi} 
to a class on a compact K\"ahler manifold. 
The proof gives another proof of Proposition \ref{ com-posi} 
without the approximation of the positive part $P$ by $\QQ$-divisors.

\begin{prop}\label{ div}
Let $\A = P + \{N\}$ be the divisorial Zariski decomposition of $\A$.
Assume that $W$ is not contained in $E_{nK}(\A)$. 
Then $W$ is not contained in $E_{nK}(P)$ and 
the equality $\Vm{W}(\A) = \Vm{W}(P) $ holds.
\end{prop}

\begin{proof}
The proposition is based on the following fact. 
Positive currents in $\A$ and positive currents in $P$ are identified  
by the correspondence $T \longmapsto   T-[N]$.
First we show the following claim. 
\begin{claim}\label{ aug}
We have $E_{nK}(\A) = E_{nK}(P)$.
\end{claim}
\begin{proof}
For a point $x \not\in E_{nK}(\A) $, there is a K\"ahler current $T$ in $\A$ with analytic singularities 
such that $T$ is smooth at $x$. 
Note $T-[N]$ is a K\"ahler current since $T$ is a K\"ahler current in $\A$. 
In fact,  $T \geq  \omega $ for some K\"ahler form $\omega$. 
Then the negative part of the Siu decomposition of $T - \omega$ still contains $[N]$. 
It yields $T -[N] \geq  \omega$ . 
Therefore $T-[N]$ is a K\"ahler current in $P$. 
We can easily see that the support of $N$ is contained in $E_{nK}(\A)$. 
Since $x$ is not contained in the support of $N$, the K\"ahler current $T-[N]$ is smooth at $x$.
Thus $x$ is not contained in $E_{nK}(P)$.

Conversely we take a point $x \not\in E_{nK}(P)$. 
Then there is a K\"ahler current $S$ in $P$ such that $S$ is smooth at $x$ .
We may assume that $S\geq \omega $.
We shall show that $x$ is not contained in the support of $N$.
To prove this, we consider the surjective map:
\begin{equation*}
\big\{ \mathrm{smooth\ real\ \textit{d}\textit{-}closed\ (1,1)\textit{-}form} \big\} \longrightarrow  H^{1,1} (M,\mathbb{R}),
\ \ \theta  \mapsto \big\{ \theta \big\}.
\end{equation*}
We regard the space of smooth real $d$-closed $(1,1)$-forms 
as the topological space 
with the Fr\'echet topology. 
For a smooth $(n-1,n-1)$-form $\gamma $, the integral 
$\int_{M}{\theta _{k}\wedge \gamma } $ tends to $ \int_{M}{\theta  \wedge \gamma} $ 
if $\theta _{k}$ converges to $ \theta $ in the Fr\'echt topology.
Hence it follows that 
the above map $\theta  \longmapsto \big\{ \theta \big\}$ is continuous 
from the duality theorem.
Thus the  map is an open map from the open mapping theorem .

Since the map is an open map, for a positive number $\e$, 
there is a sufficiently small $\delta >0$ such that 
$\delta c_{1}(N)$ contains a smooth form $\eta $ with $-\e \omega \leq  \eta \leq \e\omega $. 
Since $S$ is a K\"ahler current, 
$S+\eta + (1-\delta)[N]$ is still a positive current for a sufficiently small $\e$. 
Further the current belongs to the class $\A$. 
Now the Lelong number of $S+\eta + (1-\delta)[N]$ at $x$ is 
equal to the Lelong number of $(1-\delta)[N]$ since $S$ and $\eta $ are 
smooth at $x$.
If $\nu([N], x)$ is positive, it is a 
contradiction to the construction of $N$. 
(Recall $N=\sum \nu(T_{\min}, E)E$, where $T_{\min}$ is a current with 
minimal singularities.)
Thus $x$ is not contained in the support of $N$.
It implies that the K\"ahler current $S+[N]$ is smooth at $x$.
Hence $x$ is not contained in $E_{nK}(\A)$.
\end{proof}

Finally, we prove $\Vm{W}(\A) = \Vm{W}(P) $. 
Note that we can define the restricted volume of $P$ thanks to the claim above. 
Since the support of the current $ [N] \vert_{W_{\reg}} $ is contained in $N\cap W$, 
the absolutely continuous part of $[N] \vert_{W_{\reg}}$ is zero. 
It implies that $ [N] \vert_{W_{\reg}} $ does not affect the integration on $W$.
Therefore it follows Proposition \ref{ div} from the correspondence 
between positive currents in $\A$ and in $P$. 
\end{proof}
The following theorem says that 
Fujita's approximation theorem for the restricted volume of 
a class holds.
It leads to the continuity of the restricted volume.

\begin{theo}\label{ fujita}

The restricted volume of a class $\A$ along $W$ can be approximated 
by self-intersection numbers of semi-positive classes.
That is, the following equality holds.
\begin{equation*}
\Vm{W} (\A) = \sup_{\pi^{*}T =B +[E]} \big( \{B \} ^{d} \cdot \widetilde{W} \big), 
\end{equation*}
where the supremum is taken over all resolutions 
$\pi : \widetilde{M}\longrightarrow M$ of positive currents
 $T \in \A$ with analytic singularities such that $\pi$ is an isomorphism at a generic point of $W$
and $\widetilde{W} \not\subseteq \mathrm{Supp}(E)$.
(Here $\widetilde{W}$ denotes the strict transformation of $W$.)
\end{theo}

\begin{proof}
Let $T$ be a positive current with analytic singularities in the class $\A$ 
whose singular locus does not contain $W$.
Then we take a modification $\mu$ such that $\mu^{*}T = B + [E]$ 
and $\mu$ is an isomorphism at a generic point on $W$. 
Lemma \ref{ bir} yields 
\begin{align*}
\int_{W_{\reg}}{T\vert_{W\reg}} &= \int_{\widetilde{W}}{ (\mu^{*}T \vert_{\widetilde{W}})_{\ac}^{d}} \\
&= \int_{\widetilde{W}}{\big( (B + [E]) \vert_{\widetilde{W}} \big)_{\ac}^{d}} \\
&= \int_{\widetilde{W}}{ B ^{d}} = \big( \{B \} ^{d} \cdot \widetilde{W} \big).
\end{align*}
Therefore we obtain $\Vm{W} (\A) =\sup_{\pi^{*}T =B +[E]} 
\big(\{B\} ^{d} \cdot \widetilde{W} \big)$ 
from the definition of the restricted volume of $\A$ along $W$.
\end{proof}
In order to show the continuity of the restricted volume, 
we consider the \lq\lq domain" of the restricted volume 
for a given analytic subset $W$ on $M$. 
Further, we prove the convexity of the domain and log concavity 
of the restricted volume.

\begin{defi}
For an irreducible analytic subset $W$ on $M$, 
\textit{the domain of the restricted volume} is defined to be 
$\mathrm{Big}^{W}(M) := \big\{ \B  \in H^{1,1} (M,\mathbb{R}) \ \big|\ W \not\subseteq  E_{nK}(\B) \big\}$.
\end{defi}

\begin{prop}\label{ concave}

$(1)$\ \ $\mathrm{Big}^{W}(M)$ is an open convex set in $H^{1,1} (M,\mathbb{R})$. \\
$(2)$\ \ For $\B_{1}, \B_{2} \in \mathrm{Big}^{W}(M)$, we have 
$$ \Vm{W}(\B_{1}+\B_{2})^{1/d} \geq \Vm{W}(\B_{1})^{1/d} + \Vm{W}(\B_{2})^{1/d}.$$ 

\end{prop}
\begin{proof}
(1)\ \ 
The convexity is easily proved from $E_{nK}(\B + \B^{'}) \subseteq 
E_{nK}(\B^{'}) \bigcup E_{nK}(\B)$ 
and $E_{nK}(\B) = E_{nK}(k\B)$ for $k > 0$.
For a given class $\B$, we can see 
$E_{nK}(\B^{'}) \subseteq E_{nK}(\B)$ for every class $\B^{'}$ in a suitable open neighborhood of $\B$ 
by using the argument in Lemma \ref{ aug}.
It asserts the domain is an open set.

$(2)$\ \ In his paper \cite{Bou02}, 
Boucksom showed the log concavity for the volume of a transcendental class.
Hence it follows the log concavity of the restricted volumes of nef classes 
from Proposition \ref{ nef}. 
By using Proposition \ref{ fujita}, we can conclude that 
the restricted volume has the log concavity on $\mathrm{Big}^{W}(M)$.
\end{proof}

\begin{cor}\label{ conti}
The following map is continuous. 
\begin{align*}
 \Vm{W} (\cdot) : \mathrm{Big}^{W}(M)  \longrightarrow  \mathbb{R}, \ \ \ \ \  
 \B  \longmapsto   \Vm{W}(\B)
\end{align*}
\end{cor}
\begin{proof}
It is known fact that a concave function 
on an open convex set in $\mathbb{R}^{N}$ is continuous.
Therefore Corollary \ref{ conti} follows from Proposition \ref{ concave}.
\end{proof}

\subsection{Proof of Theorem \ref{ main3}}
In this subsection, we prove Theorem \ref{ main3} 
by using the analytic description of the restricted volume with currents. 
It gives another proof of  Theorem \ref{ main}. 
Let $\A \in H^{1,1}(X,\RR)$ be a big class on a smooth projective variety 
$X$ and 
$\A = P + \{N \}$ the divisorial Zariski decomposition. 
We have $E_{nK} (\A) = E_{nK} (P)$ by Lemma \ref{ aug}.
Hence we can consider the restricted volume of $P$ along $V$.

The strategy of the proof is essentially same as  Theorem \ref{ main}.
From Proposition \ref{ div}, 
we have $\vol(\A) =\vol(P) $ for an irreducible subvariety $V$ on $X$
such that $V\not\subseteq E_{nK}(\A)$. 
Moreover, 
Proposition \ref {nef} says that $ \vol (P) = (V\cdot P ^{d} )$ holds if $P$ is nef.
Hence when $\A$ admits a Zariski decomposition, 
the restricted volumes along any cohomologus subvarieties coincide.

Let us show that condition (3) implies condition (1).
We assume the non-nef locus $E_{nn}(P)$ is not empty for a contradiction and 
fix a very ample divisor $A$ on $X$. 
Then there are smooth curves $C$ and $C^{'}$ with the following properties:\\
\ \ \ (1)\ \ \ $C^{'}$ does not intersect with the non-nef locus $E_{nn}(P)$.\\
\ \ \ (2)\ \ \ $C$ and $C^{'}$ are not contained in the non-K\"ahler locus $E_{nK}(\A)$. \\
\ \ \ (3)\ \ \ $C$ intersects with the non-nef locus $E_{nn}(P)$ at $x_{0} \in X$.\\
\ \ \ (4)\ \ \ $C$ and $C^{'}$ are complete intersections of members of the complete linear system of $A$.\\
Then we prove the following lemma 
for a contradiction. 

\begin{lemm}\label{ estimate2} 
In the situation above, the followings hold. \\
\ \ \ \ \ \ \ $(A)$ \ \ $\V{C^{'}}(\A) = (C^{'}\cdot P), $ \ \ \ \ \ 
$(B)$ \ \ $\V{C}(\A) < (C\cdot P).$ 
\end{lemm}
\begin{proof}
From the definition of the restricted volume of $P$ and Proposition \ref {div}, 
we obtain
\begin{equation*}
 \V{C}(\A) = \V{C} (P) = \sup_{T \in P} \int_{C}{( T\vert_{C} ) _{\ac}} 
\end{equation*}
Here $T$ runs through positive currents with analytic singularities in the class $P$ 
whose singular loci do not contain $C$.
Now $T \vert_{C}$ is also a positive current with analytic singularities. 
In general, the Siu decomposition coincides with the Lebesgue decomposition 
for a $d$-closed positive current with analytic singularities. 
Therefore we have 
${ (T \vert_{C})_{\ac} = T \vert_{C} -\sum_{x \in C} \nu(T \vert_{C} , x)[x] }$. 
On the other hand, we have
\begin{equation*}
\int_{C} {T \vert_{C}} =  (C\cdot P). 
\end{equation*}
In fact, we can easily see
\begin{equation*}
\int_{C} {T \vert_{C}}  = (C\cdot P) + \int_{C} {\ddbar \varphi \vert_{C}}, 
\end{equation*}
where $\varphi$ is an $L^{1}$-function on $X$ such that 
$T  = \theta + \ddbar \varphi $. 
Here $\theta$ denotes a smooth $(1,1)$-form in $P$. 
By applying the approximation theorem (Theorem \ref{ app-dem82}) to $\varphi \vert_{C}$, we obtain smooth functions $\varphi _{k}$ on $C$ 
such that $\ddbar \varphi _{k} $ weakly converges to $\ddbar \varphi \vert_{C}$ . 
Thus $\int_{C}{\ddbar \varphi _{k}} $ tends to $ \int_{C} {\ddbar \varphi \vert_{C}}$.
On the other hand, $\int_{C}{\ddbar \varphi _{k}}$ is equal to zero for every $k$  
from Stokes's theorem. (Note that $\ddbar \varphi _{k}$ is smooth.) 

Hence we obtain 
\begin{align*}
\V{C}(\A) &= \sup_{T \in c_{1}(P)} \{ (C \cdot P) - \sum_{x \in C} \nu(T \vert_{C} , x)\}  \\
& = (C \cdot P) - \inf_{T \in P} \sum_{x \in C} \nu(T \vert_{C} , x).
\end{align*}

In general, 
the Lelong number of the restriction of a current is more than or 
equal to the Lelong number of the current.
Further, 
$\nu(T_{\mathrm{min}}, x) \leq \nu(T , x) $ holds from the definition of a current with minimal singularities. 
Therefore we obtain
\begin{equation*}
\V{C}(\A)  \leq (C \cdot P) -\sum_{x \in C} \nu(T_{\mathrm{min}} , x). 
\end{equation*}
The curve $C$ intersects with the non-nef locus $E_{nn}(P)$ at $x_{0}$ from property (3). 
Hence $\nu(T_{\mathrm{min}} , x_{0})$ is positive.
It implies $\V{C}(\A)  \leq (C \cdot P) - \nu(T_{\mathrm{min}} , x_{0}) < (C \cdot P) $.
Here $T_{\min}$ is a current with minimal singularities in $P$. 
Therefore inequality (B) holds. 

Finally we shall prove equality (A).
By the first half argument, we have $\V{C^{'}}(\A) \leq (C^{'} \cdot P) $.
To prove the converse inequality,  
we take a K\"ahler current $S $ with analytic singularities in $\A$.
We may assume $S \geq \omega$, where $\omega$ is a K\"ahler form on $X$. 
By applying the approximation theorem (Theorem \ref{ app-bou02}) to a current 
$T_{\min}$ with minimal singularities in $P$, 
we obtain positive currents $T_{k}$ with analytic singularities with the following properties.
\\
\ \ \ $\mathrm{(b')}$\ \ \ $T_{k} \geq -\e _{k} \omega$ and $\e_{k}$ converges to zero.\\
\ \ \ $\mathrm{(c')}$\ \ \ The Lelong number $\nu(T_{k},x)$ increases to
$\nu(T_{\mathrm{min}} , x)$ for every point $x \in X$.\\
For every positive number $\delta $, there is $k(\delta )$ such that 
$(1-\delta ) T_{k(\delta)} + \delta S$ is a positive current. 
Since  $(1-\delta ) T_{k(\delta)} + \delta S$ is 
a positive current with analytic singularities, 
the inequality 
\begin{equation*}
\V{C^{'}} (\A) \geq \int_{C^{'}} \big( ((1-\delta ) T_{k(\delta)} + \delta S )\vert_{C^{'}} \big)_{\ac}
\end{equation*}
holds by the definition of the restricted volume. 
The Lelong number of $T_{k}$ at every point in $C$
is zero  by property (3). 
It implies $T_{k}$ is smooth on $C$.
Thus we obtain 
\begin{equation*}
\V{C^{'}} (\A) \geq  
(1-\delta) (C^{'} \cdot P) - \delta 
\int_{C^{'}} \big(  S \vert_{C^{'}} \big)_{\ac}.
\end{equation*} 
for every $\delta$. 
When $\delta$ tends to zero, we obtain $\V{C^{'}} (\A) \geq (C^{'} \cdot P) $. 
\end{proof}

\end{document}